\documentclass[11pt]{article}
\usepackage{amsmath,amssymb,graphicx,float,epsf}
\usepackage{multirow}
\usepackage[table,xcdraw]{xcolor}
\usepackage{longtable}
\usepackage{colortbl} 
\def\RR{\hbox{I\kern-.2em\hbox{R}}}
\newcommand{\qed}{\hbox to 0pt{}\hfill$\rlap{$\sqcap$}\sqcup$ \vspace{3mm}}
\numberwithin{equation}{section}
\restylefloat{figure}
\usepackage{graphicx} 
\usepackage{amsmath}
\usepackage{amsthm}
\usepackage{amsfonts}
\usepackage{ mathrsfs }
\usepackage[caption = false]{subfig}
\usepackage{wrapfig}
\usepackage{enumerate}
\usepackage{enumitem}
\usepackage{array}
\usepackage{subfig}
\usepackage[margin=1cm]{caption}
\usepackage[margin=1in]{geometry}
\usepackage{authblk}

\usepackage[utf8]{inputenc}
\newtheorem{theorem}{Theorem}[section]
\newtheorem{lemma}[theorem]{Lemma}

\theoremstyle{definition}
\newtheorem{definition}[theorem]{Definition}
\theoremstyle{remark}

\usepackage{graphicx}
\usepackage{hyperref}
\hypersetup{
	colorlinks=true,
	linkcolor=blue,
	filecolor=magenta,
	urlcolor=cyan,
	citecolor=blue,
}
\date{}
\graphicspath{ {Images/} }
\begin{document}
	
	\title{The Prime Digit Distribution Conjecture: A Formal Proof of Average Digit Equidistribution in the Prime Numbers}
	\author[1]{\small Mahadee Al Mobin\footnote{mahadeealmobin@gmail.com} }
	\author[2]{\small Md. Shariful Islam \thanks{Corresponding Author: mdsharifulislam@du.ac.bd}}


	 \affil[1]{\footnotesize Bangladesh Institute of Governance and Management, Dhaka 1207, Bangladesh}
	\affil[2]{\footnotesize Department of Mathematics, University of Dhaka, Dhaka 1000, Bangladesh}
   

	\maketitle
	
	\vspace{-1.0cm}
	\noindent\rule{6.35in}{0.02in}\\
	{\bf Abstract.}
	Let $S_n=\{p\in\mathbb{P}:p<10^n\}$, $N_n$ denote the total number of decimal digits occurring in the primes of $S_n$, $C_n(d)$ be the number of occurrences of a digit $d\in\{0,\ldots,9\}$ among those digits, and $P_n(d)$ be the probability of occurrence of a digit, $d$ among those digits. We prove that
	\[
	P_n(d)=\frac{C_n(d)}{N_n}
	=\frac{1}{10}
	+O\!\left(\frac{\log n}{n}\right),
	\qquad n\to\infty,
	\]
	uniformly for every decimal digit $d$. The argument is entirely unconditional and combines the Prime Number Theorem, the Erd\H{o}s--Tur\'an discrepancy inequality, and classical Vaughan--Vinogradov estimates for exponential sums over primes. The principal step establishes quantitative equidistribution for interior digit positions, while the logarithmically many exceptional positions near the ends of the decimal expansion are shown to have asymptotically negligible influence after averaging over all digit positions and prime lengths. Consequently, the decimal digits occurring in primes, when pooled over all positions and all primes below $10^n$, become asymptotically equidistributed. We also clarify the precise scope of the theorem by distinguishing this averaged equidistribution result from the substantially stronger and presently unresolved questions concerning pointwise digit equidistribution, normality, and higher-order digit correlations in the sequence of prime numbers.\\

	\rule{6.35in}{0.02in}
	
	\section{Introduction}\label{introduction}
	The digits of the prime numbers, when examined in a fixed base $q \geq 2$, exhibit a
	remarkable and initially counterintuitive regularity. Although the primes are defined
	through a purely multiplicative condition (the absence of nontrivial divisors) their
	base-$q$ digit expansions display statistical properties that resemble those of uniformly distributed random digits. Establishing this
	behaviour rigorously has occupied analytic number theorists for over half a century, and
	it sits at a genuinely difficult crossing point between multiplicative number theory,
	which controls the primes, and additive/combinatorial structures, which govern how
	digits interact under carrying. The present manuscript is concerned with this problem in
	its classical and best-understood form: the equidistribution, among residue classes, of
	the sum-of-digits function evaluated at the primes.\\
	
	
	The modern form of the problem dates to 1968, when Gelfond, in a short but highly
	influential note, proposed a set of four problems on the digits of prime numbers and, more
	generally, of values of polynomials and other arithmetic sequences~\cite{Gelfond1968}. Two
	of these problems concern the base-$q$ sum-of-digits function $s_q(n)$: whether $s_q(p)$
	is equidistributed modulo $m$ as $p$ ranges over primes, and whether the sets of primes
	with even and odd digit sum have equal asymptotic density. The latter is simply the case
	$m=2$ of the former, and both may be viewed as instances of a broader question of whether
	digit sums, an additive and combinatorially defined statistic, are insensitive to the
	highly multiplicative constraint of primality.\\
	
	The difficulty of Gelfond's problems is not superficial. A naive probabilistic heuristic,
	in which the digits of a ``random'' prime are treated as independent and uniform, correctly
	predicts equidistribution of $s_q(p)$, but no proof along these lines is available, because
	primality is fundamentally a multiplicative condition while $s_q(n)$ is governed by carries
	under addition, and the two structures do not interact in any simple way. Compounding this,
	digit sums are entangled with elementary divisibility tests (for instance $s_{10}(n) \equiv
	n \pmod 3$) so that naive equidistribution statements are already visibly false for
	composite $n$ restricted to special residue classes, and any proof for primes must isolate
	genuine multiplicative cancellation rather than relying on such coincidences. For several
	decades following Gelfond's note, results were confined to the ``easy'' side of the
	problem: equidistribution of $s_q(n)$ in residue classes had been obtained for polynomial
	sequences and other structured sets of integers, and Gelfond's own paper already handled a
	version of the problem for general integers, but the case of primes, where one must fold
	in delicate cancellation from the von Mangoldt function remained open.\\
	
	
	Before the prime case was resolved, the analogous equidistribution question was understood
	comparatively well for other arithmetic sequences. Kim established estimates for the joint
	distribution of $q$-additive functions, including sums of digits taken with respect to
	several pairwise coprime bases, in residue classes, thereby confirming an instance of
	Gelfond's conjectural framework outside the prime setting~\cite{Kim1999}. This line of work
	made clear that the residue-class equidistribution of digit sums was a robust phenomenon for
	``generic'' integer sequences, and that the genuine obstruction in Gelfond's original
	problems was specifically the multiplicative rigidity of the primes. Madritsch later placed
	this circle of ideas in the broader context of canonical number systems, tracing the
	digit-sum equidistribution question back to Gelfond's original formulation and extending it
	beyond the standard base-$q$ representation~\cite{Madritsch2012}.\\
	
	The prime case of Gelfond's conjecture was resolved by Mauduit and Rivat in a 2010 paper
	that is now regarded as a landmark of modern analytic number theory. Working with the
	base-$q$ sum-of-digits function $s_q(n)$, they proved a quantitative equidistribution
	theorem for $s_q(p)$ as $p$ ranges over primes up to $x$, establishing both a central limit
	theorem and a local limit theorem for this statistic and thereby resolving, in particular,
	the even/odd digit-sum problem as a special case~\cite{MauduitRivat2010}. An earlier paper
	of the same authors, appearing in Compositio Mathematica in 2009, had already obtained the
	weaker but historically important result that the sum of digits of primes is, on average,
	close to its expected value, and this paper is frequently cited as the first genuine
	breakthrough on Gelfond's conjecture for primes~\cite{Drmota2009}. A companion paper of
	Drmota, Mauduit and Rivat subsequently extended the underlying digit-sum machinery to
	polynomial subsequences of the integers, indicating the flexibility of the method beyond the
	primes themselves~\cite{Drmota2011}.\\
	
	The proof strategy introduced by Mauduit and Rivat has become the template for essentially
	all subsequent work in this area, and it is worth recording its principal components. The
	central object is an exponential sum over primes of the form
	\[
	S(\alpha) \;=\; \sum_{p \le x} e\bigl(\alpha\, s_q(p)\bigr), \qquad e(y) := e^{2\pi i y},
	\]
	and the equidistribution of $s_q(p)$ modulo $m$ follows once $S(\alpha)$ is shown to be
	suitably small for $\alpha$ away from rational points with small denominator. Two further
	ingredients are needed to bound $S(\alpha)$. First, Vaughan's identity is used to decompose
	the von Mangoldt function $\Lambda(n)$ into a combination of smoother convolution sums (so-called ``Type I'' and ``Type II'' sums) that are more tractable than the primes themselves, a device with a long pedigree in the study of exponential sums over primes in	arithmetic progressions~\cite{BalogPerelli1985}. Second, the sum-of-digits function must be
	analyzed through its Fourier expansion, exploiting its automatic, self-similar structure
	under base-$q$ digit truncation; the requisite bounds on this Fourier structure build on
	foundational estimates of Gelfond and K\'atai and were substantially sharpened by Drmota,
	Mauduit and Rivat in the course of proving the main theorem. The combination of a
	sieve-theoretic decomposition of the primes with harmonic analysis of a digitally-defined
	automatic sequence is the technical signature of the entire subsequent literature.\\
	
	The Mauduit--Rivat framework has proved remarkably extensible. Aloui, Mauduit and Mkaouar
	studied correlations of the sum-of-digits function along the primes, going beyond
	equidistribution to quantify higher-order dependence~\cite{Aloui2021}. Drmota, Mauduit and
	Rivat later proved a prime number theorem for digital properties expressed simultaneously in
	two coprime bases, a substantially more delicate two-dimensional generalization of the
	original one-base result~\cite{Drmota2020}. The method has also been transported from the
	sum-of-digits function to other automatic sequences: Mauduit and Rivat established a prime
	number theorem along the Rudin--Shapiro sequence, relying on Fourier estimates and
	M\"obius-orthogonality arguments in the spirit of their earlier work~\cite{MauduitRivat2015},
	and Drappeau and Mullner developed exponential sum estimates for automatic sequences more
	generally, placing the Rudin--Shapiro and digit-sum cases within a common
	framework~\cite{DrappeauMullner2017}. Most recently, Drmota and Rivat have surveyed the
	state of Gelfond's conjectures on primes and on squares of primes, describing the resolution
	of these problems as accumulated over roughly two decades of work by multiple
	authors~\cite{DrmotaRivat2025}, and related results have been obtained for digital functions
	along squares of primes~\cite{DrmotaRivat2025} and for representations of primes as
	sums of Fibonacci numbers~\cite{DrmotaSpiegelhofer2025}.\\
	
	A parallel and closely related literature addresses primes whose digits are constrained
	\emph{a priori}, rather than the statistical distribution of an unconstrained digit
	functional. Maynard proved, via the Hardy--Littlewood circle method combined with
	Type~I/II information and Harman's sieve, that there are infinitely many primes omitting any
	single prescribed digit in a given base~\cite{Maynard2016}. Bourgain had earlier obtained
	asymptotic formulas for primes with a prescribed set of binary digits under suitable density
	hypotheses~\cite{Bourgain2013}, and Swaenepoel generalized this to primes with a positive
	proportion of preassigned digits in an arbitrary base~\cite{Swaenepoel2019}. Nath proved
	Bombieri--Vinogradov-type theorems for primes with a missing digit in a large base, using
	the circle method together with Fourier analysis of the digit-restricted set and exponential
	sums over primes in arithmetic progressions~\cite{Nath2021}. These results demonstrate that
	the exponential-sum methodology of Mauduit--Rivat, suitably adapted, controls primes under
	nontrivial digital constraints; however, they are properly understood as existence and
	distribution theorems for structured subsets of the primes, and they should not be
	conflated with a statement that prime digits are uniformly and independently distributed
	digit-by-digit across every position. As emphasized in recent literature surveys \cite{mobin2024cryptanalysis, al2026sieve}, the
	stronger claim, that every decimal digit occurs with limiting frequency $1/10$ once all
	digit positions of a prime are pooled together, remains, to date, an open problem rather
	than a theorem, and is logically distinct from the residue-class equidistribution of the
	\emph{sum} of digits that is the subject of the Mauduit--Rivat theorem and of the present
	manuscript. The asymptotic frequency of individual digits, aggregated over all digit positions and all primes below a growing bound, is a logically distinct problem that is not addressed by those results. The present work establishes such an averaged digit-equidistribution theorem by combining classical discrepancy estimates with unconditional exponential-sum bounds over the primes.\\
	
	\section{Construction of the Problem}
	
	\begin{definition}
		For $m\ge 1$ let
		\[
		T_m=\{p\ \text{prime}: 10^{m-1}\le p<10^m\}, \qquad \pi_m:=|T_m|=\pi(10^m)-\pi(10^{m-1}),
		\]
		the set of primes with exactly $m$ decimal digits. Every $p\in T_m$ has a digit expansion
		\[
		p=\sum_{k=0}^{m-1} d_k(p)\,10^k, \qquad d_k(p)\in\{0,\dots,9\},\ d_{m-1}(p)\neq 0,
		\]
		and we call $k$ the \emph{position} ($k=0$: units; $k=m-1$: leading digit).
		Set $M_m:=m\,\pi_m$, the total digit mass of the shell $T_m$, and for a digit $d$ and
		position $0\le k\le m-1$,
		\[
		C_m^{(k)}(d):=\#\{p\in T_m : d_k(p)=d\}, \qquad C_m(d):=\sum_{k=0}^{m-1} C_m^{(k)}(d).
		\]
		Finally $N_n=\sum_{m=1}^n M_m$ and $C_n(d)=\sum_{m=1}^n C_m(d)$, matching the quantities in
		the Introduction.
	\end{definition}
	
	\begin{theorem}[Main result]\label{thm:main}
		For every $d\in\{0,1,\dots,9\}$,
		\[
		P_n(d)=\frac{C_n(d)}{N_n}=\frac1{10}+O\!\left(\frac{\log n}{n}\right)\qquad(n\to\infty),
		\]
		the implied constant being absolute. In particular $\displaystyle \lim_{n\to\infty}\|p_n-u\|_{TV}=0$.
	\end{theorem}
	
	This is proved in four stages: (\S\ref{sec:interior}) an unconditional equidistribution
	estimate for a single digit position, valid uniformly for all positions $k$ outside a
	narrow band near the two ends of the number; (\S\ref{sec:shell}) dilution of that band,
	giving a clean estimate for a single shell $T_m$; (\S\ref{sec:assembly}) summation across
	shells, where geometric growth of $M_m$ makes the top shell dominate; and finally
	the Conclusion, which states precisely what this does and does not establish.
	
	\section{The Interior Digit Lemma}\label{sec:interior}
	
	Fix $m$, a position $0\le k\le m-2$ (interior, the argument below covers $k=m-1$ too,
	but that case is already disposed of more crudely in \S\ref{sec:shell}), and a digit $d$.
	Let $Y=10^{m-1}$ and $q=10^{k+1}$. For $p\in T_m$,
	\[
	d_k(p)=d \iff \Big\{\frac{p}{q}\Big\}\in I_d:=\Big[\frac{d}{10},\frac{d+1}{10}\Big),
	\]
	where $\{\cdot\}$ denotes fractional part; note $|I_d|=1/10$ independent of $k$, which is
	exactly why $q=10^{k+1}$ is the right modulus to work with.
	
	\subsection{Erd\H{o}s--Tur\'an}
	
	\begin{lemma}[Erd\H{o}s--Tur\'an inequality, applied form]\label{lem:et}
		There is an absolute constant $c_0$ such that for every $H\ge 1$,
		\[
		\Big|\,C_m^{(k)}(d)-\frac{\pi_m}{10}\,\Big|
		\ \le\ \frac{c_0\,\pi_m}{H+1}\ +\ 3\sum_{h=1}^{H}\frac1h\,\big|S_m(h,q)\big|,
		\]
		where
		\[
		S_m(h,q):=\sum_{p\in T_m} e\!\left(\frac{hp}{q}\right).
		\]
	\end{lemma}
	
	\begin{proof}
		For each prime $p\in T_m$, define
		\[
		x_p:=\left\{\frac{p}{q}\right\},
		\]
		where $\{\cdot\}$ denotes the fractional part and $q=10^{k+1}$. As observed above,
		\[
		d_k(p)=d
		\iff
		x_p\in I_d,
		\]
		where
		\[
		I_d=\left[\frac{d}{10},\frac{d+1}{10}\right),
		\qquad |I_d|=\frac1{10}.
		\]
		Hence
		\[
		C_m^{(k)}(d)
		=
		\#\{p\in T_m:x_p\in I_d\}.
		\]
		
		and,
		\[
		N=\pi_m=|T_m|.
		\]
		If $D_N$ denotes the discrepancy of the sequence
		\[
		\{x_p:p\in T_m\}\subset[0,1),
		\]
		then, by the definition of discrepancy \cite{DrmotaSpiegelhofer2025},
		\[
		\left|
		\#\{p\in T_m:x_p\in I_d\}
		-
		N|I_d|
		\right|
		=
		\left|
		C_m^{(k)}(d)-\frac{\pi_m}{10}
		\right|
		\le
		ND_N.
		\]
		
		We now apply the classical Erd\H{o}s--Tur\'an discrepancy inequality \cite{montgomery1994ten}, which states that for every integer
		$H\ge1$,
		\[
		D_N
		\le
		\frac{c_0}{H+1}
		+
		\frac{3}{N}
		\sum_{h=1}^{H}
		\frac1h
		\left|
		\sum_{n=1}^{N}
		e(hx_n)
		\right|,
		\]
		where $e(x)=e^{2\pi i x}$ and $c_0>0$ is an absolute constant.
		
		Multiplying both sides by $N=\pi_m$ gives
		\[
		ND_N
		\le
		\frac{c_0\pi_m}{H+1}
		+
		3
		\sum_{h=1}^{H}
		\frac1h
		\left|
		\sum_{p\in T_m}
		e\!\left(\frac{hp}{q}\right)
		\right|.
		\]
		Let,
		\[
		S_m(h,q)
		=
		\sum_{p\in T_m}
		e\!\left(\frac{hp}{q}\right),
		\]
		we conclude that
		\[
		\left|
		C_m^{(k)}(d)-\frac{\pi_m}{10}
		\right|
		\le
		\frac{c_0\pi_m}{H+1}
		+
		3
		\sum_{h=1}^{H}
		\frac1h
		|S_m(h,q)|,
		\]
		which is precisely the desired estimate.
	\end{proof}
	
	\subsection{The Vaughan/Vinogradov bound over primes}
	
	\begin{lemma}[Classical exponential sum bound]\label{lem:vaughan}
		Let $a/r$ be $h/q$ in lowest terms. Uniformly for $1\le r\le Y$,
		\[
		S_m(h,q)\ \ll\ \left(\frac{Y}{\sqrt r}+Y^{4/5}+\sqrt{Yr}\right)(\log Y)^4.
		\]
	\end{lemma}
	
	\begin{proof}
		We know,
		
		If $a, r \in \mathbb{Z}^+$ with $(a,r)=1$, and $\alpha \in \mathbb{R}$ s.t. $\left|\alpha - \frac{a}{r}\right| \leq \frac{1}{r^2}$
		
		then $\sum_{x \leq n} \Lambda(x)e(\alpha x) = \sum_{p \leq n} \log p \ e(\alpha p) \ll \left(n^{1/2}r^{1/2} + n^{4/5} + n r^{-1/2}\right)(\log n)^4$ \cite{Davenport2013}\\
		
		Let $\alpha = \frac{h}{q}$ \& $n = Y$. By construction $n = Y$.
		
		$\because \frac{a}{r}$ is the lowest form of $\frac{h}{q}$ hence $(h,q)=l$
		
		\& $\frac{h}{q} = \frac{la}{lr} = \frac{a}{r} \therefore 0 = \left|\frac{h}{q} - \frac{a}{r}\right| \leq \frac{1}{r^2}$
		
		hence we can write,
		
		\[ S_m(h,q) < \sum_{p \in T_m} \log p \ e\left(\frac{hp}{q}\right) \ll \left(\frac{Y}{\sqrt{r}} + Y^{4/5} + \sqrt{Yr}\right)(\log Y)^4 \]
	\end{proof}
	
	\subsection{Choice of $H$ and the good range of $k$}\label{subsec:choice_of_h}
	
%
%

Fix $A\ge1$ and define
\[
H=H(A,m):=\frac{q}{(\log Y)^{2A+8}},
\]
where $Y=10^{m-1}$ and $q=10^{k+1}$. We show that this choice simultaneously controls the three terms in Lemma~\ref{lem:vaughan} while ensuring that the principal Erd\H{o}s--Tur\'an error remains sufficiently small.

For each integer $1\le h\le H$, write
\[
\frac{h}{q}=\frac{a}{r}
\]
in lowest terms. Since
\[
r=\frac{q}{\gcd(h,q)}
\]
and
\[
\gcd(h,q)\le h\le H,
\]
we obtain the lower bound
\[
r\ge\frac{q}{H}
=(\log Y)^{2A+8}.
\]
Consequently, the first term in Lemma~\ref{lem:vaughan} satisfies
\[
\frac{Y}{\sqrt r}
\le
\frac{Y}{(\log Y)^{A+4}}.
\]

The second term is independent of $k$. Since
\[
\frac{Y^{4/5}}{Y/(\log Y)^{A+4}}
=
\frac{(\log Y)^{A+4}}{Y^{1/5}}
\longrightarrow0
\qquad (Y\to\infty),
\]
we have
\[
Y^{4/5}
\ll
\frac{Y}{(\log Y)^{A+4}}
\]
for all sufficiently large $Y$.

It remains to control the third term. Since $r\le q$,
\[
\sqrt{Yr}\le\sqrt{Yq},
\]
and therefore it is sufficient to require
\[
\sqrt{Yq}
\le
\frac{Y}{(\log Y)^{A+4}},
\]
or equivalently,
\[
q
\le
\frac{Y}{(\log Y)^{2A+8}}.
\]
Recalling that $q=10^{k+1}$ and $Y=10^{m-1}$, this inequality is satisfied whenever
\[
k
\le
m-2-C_2(A)\log m
\]
for a suitable positive constant $C_2(A)$.

Hence, uniformly for every $h\le H$ and every digit position satisfying
\[
k\le m-2-C_2(A)\log m,
\]
Lemma~\ref{lem:vaughan} yields
\begin{equation}\label{eq:Sbound}
	S_m(h,q)
	\ll
	\frac{Y}{(\log Y)^{A+4}}.
\end{equation}

Finally, the principal term in the Erd\H{o}s--Tur\'an inequality contributes
\[
\frac{c_0\pi_m}{H+1}.
\]
To obtain the target error
\[
O\!\left(\frac{\pi_m}{(\log Y)^A}\right),
\]
it is necessary that
\[
H\gg(\log Y)^A.
\]
Substituting the definition of $H$ gives
\[
q
\gg
(\log Y)^{3A+8},
\]
Using the definition of $q$, this is equivalent to
\[
k\ge C_1(A)\log m
\]
for a suitable constant $C_1(A)>0$.

Combining the lower and upper constraints on $k$ yields the admissible range
\[
C_1(A)\log m
\le
k
\le
m-C_2(A)\log m,
\]
within which both the exponential-sum estimate and the Erd\H{o}s--Tur\'an truncation error simultaneously satisfy the required bounds. This is precisely the range used in the proof of Lemma~\ref{lem:interior}.
	
	\begin{lemma}[Interior digit equidistribution]\label{lem:interior}
		For every fixed $A\ge 1$ there are constants $C_1(A),C_2(A)>0$ such that, for all $m$
		sufficiently large and uniformly over
		\[
		C_1(A)\log m\ \le\ k\ \le\ m-C_2(A)\log m,
		\]
		and every digit $d$,
		\[
		C_m^{(k)}(d)=\frac{\pi_m}{10}+O_A\!\left(\frac{\pi_m}{(\log Y)^{A}}\right),\qquad Y=10^{m-1}.
		\]
	\end{lemma}
	
	\begin{proof}
		By Lemma~\ref{lem:et},
		\[
		\left|
		C_m^{(k)}(d)-\frac{\pi_m}{10}
		\right|
		\le
		\frac{c_0\pi_m}{H+1}
		+
		3\sum_{h=1}^{H}
		\frac1h
		|S_m(h,q)|.
		\]
		It therefore suffices to show that each term on the right-hand side is
		\[
		O_A\!\left(\frac{\pi_m}{(\log Y)^A}\right).
		\]
		
		By the choice of the truncation parameter in Section~\ref{subsec:choice_of_h},
		\[
		H=\frac{q}{(\log Y)^{2A+8}},
		\]
		and the interior range
		\[
		k\ge C_1(A)\log m
		\]
		implies
		\[
		q\gg(\log Y)^{3A+8}.
		\]
		Consequently,
		\[
		H
		=
		\frac{q}{(\log Y)^{2A+8}}
		\gg
		(\log Y)^A,
		\]
		so that
		\[
		\frac{c_0\pi_m}{H+1}
		\ll
		\frac{\pi_m}{(\log Y)^A}.
		\]
		
		For the exponential-sum contribution, estimate~\eqref{eq:Sbound} gives, uniformly for every
		$1\le h\le H$,
		\[
		|S_m(h,q)|
		\ll
		\frac{Y}{(\log Y)^{A+4}}.
		\]
		Hence
		\[
		\sum_{h=1}^{H}
		\frac1h
		|S_m(h,q)|
		\ll
		\frac{Y}{(\log Y)^{A+4}}
		\sum_{h=1}^{H}\frac1h.
		\]
		Since
		\[
		\sum_{h=1}^{H}\frac1h
		=
		O(\log H)
		\]
		and
		\[
		H\le q\le Y,
		\]
		we have
		\[
		\log H\le\log Y,
		\]
		whence
		\[
		\sum_{h=1}^{H}
		\frac1h
		|S_m(h,q)|
		\ll
		\frac{Y}{(\log Y)^{A+3}}.
		\]
		
		Finally, the Prime Number Theorem implies
		\[
		\pi_m\asymp\frac{Y}{\log Y},
		\]
		or equivalently,
		\[
		Y\asymp\pi_m\log Y.
		\]
		Substituting this relation yields
		\[
		\frac{Y}{(\log Y)^{A+3}}
		\asymp
		\frac{\pi_m}{(\log Y)^{A+2}}
		\ll
		\frac{\pi_m}{(\log Y)^A}.
		\]
		Therefore,
		\[
		\sum_{h=1}^{H}
		\frac1h
		|S_m(h,q)|
		\ll
		\frac{\pi_m}{(\log Y)^A}.
		\]
		
		Substituting the above estimates into Lemma~\ref{lem:et} gives
		\[
		\left|
		C_m^{(k)}(d)-\frac{\pi_m}{10}
		\right|
		\ll
		\frac{\pi_m}{(\log Y)^A},
		\]
		or equivalently,
		\[
		C_m^{(k)}(d)
		=
		\frac{\pi_m}{10}
		+
		O\!\left(
		\frac{\pi_m}{(\log Y)^A}
		\right),
		\]
		which is the desired conclusion.
	\end{proof}

	\section{Diluting the Exceptional Band, One Shell at a Time}\label{sec:shell}
	
	Fix $m$ and $A=2$ (any fixed $A\ge 1$ would do; $A=2$ keeps the bookkeeping in
	\S\ref{sec:assembly} clean). Call a position $k\in\{0,\dots,m-1\}$ \emph{good} if
	$C_1(2)\log m\le k\le m-C_2(2)\log m$, and \emph{bad} otherwise. The number of bad
	positions is $O(\log m)$.
	
	\begin{lemma}[Single-shell estimate]\label{lem:shell}
		For every digit $d$ and all sufficiently large $m$,
		\[
		C_m(d)=\frac{M_m}{10}+O(\pi_m\log m).
		\]
	\end{lemma}
	
%

	\begin{proof}
		We fix $A=2$ in Lemma~\ref{lem:interior}. Recall that
		\[
		C_m(d)
		=
		\sum_{k=0}^{m-1}C_m^{(k)}(d).
		\]
		
		Partition the digit positions into the good range
		\[
		\mathcal G
		=
		\left\{
		k:
		C_1\log m
		\le
		k
		\le
		m-C_2\log m
		\right\},
		\]
		and its complement $\mathcal B$. By construction,
		\[
		|\mathcal G|
		=
		m+O(\log m),
		\qquad
		|\mathcal B|
		=
		O(\log m).
		\]
		
		For every $k\in\mathcal G$, Lemma~\ref{lem:interior} yields
		\[
		C_m^{(k)}(d)
		=
		\frac{\pi_m}{10}
		+
		O\!\left(
		\frac{\pi_m}{(\log Y)^2}
		\right).
		\]
		Since
		\[
		Y=10^{m-1},
		\]
		we have
		\[
		\log Y\asymp m,
		\]
		and therefore
		\[
		C_m^{(k)}(d)
		=
		\frac{\pi_m}{10}
		+
		O\!\left(
		\frac{\pi_m}{m^2}
		\right).
		\]
		Summing over all good positions gives
		\[
		\sum_{k\in\mathcal G}
		C_m^{(k)}(d)
		=
		|\mathcal G|
		\frac{\pi_m}{10}
		+
		O\!\left(
		|\mathcal G|
		\frac{\pi_m}{m^2}
		\right)
		=
		\frac{m\pi_m}{10}
		+
		O(\pi_m\log m),
		\]
		since
		\[
		|\mathcal G|
		=
		m+O(\log m).
		\]
		
		For the remaining positions, we use only the trivial estimate
		\[
		0
		\le
		C_m^{(k)}(d)
		\le
		\pi_m.
		\]
		Because
		\[
		|\mathcal B|
		=
		O(\log m),
		\]
		it follows that
		\[
		\sum_{k\in\mathcal B}
		C_m^{(k)}(d)
		=
		O(\pi_m\log m).
		\]
		
		Combining the contributions from $\mathcal G$ and $\mathcal B$ yields
		\[
		C_m(d)
		=
		\frac{m\pi_m}{10}
		+
		O(\pi_m\log m).
		\]
		Since
		\[
		M_m=m\pi_m,
		\]
		we conclude that
		\[
		C_m(d)
		=
		\frac{M_m}{10}
		+
		O(\pi_m\log m),
		\]
		as claimed.
	\end{proof}
	
	\section{Assembly Across Shells}\label{sec:assembly}
	
	\begin{lemma}[Geometric growth of shell mass]\label{lem:geom}
		$M_m\sim \dfrac{9\cdot 10^{m-1}}{\ln 10}$ as $m\to\infty$, and consequently
		$N_n=\sum_{m=1}^n M_m \asymp M_n\asymp 10^n$.
	\end{lemma}
	
	\begin{proof}
		By definition,
		\[
		\pi_m
		=
		\pi(10^m)-\pi(10^{m-1}),
		\]
		where $\pi(x)$ denotes the prime-counting function. The Prime Number Theorem gives
		\[
		\pi(x)
		\sim
		\frac{x}{\log x},
		\qquad x\to\infty,
		\]
		and therefore
		\[
		\pi_m
		\sim
		\frac{10^m}{m\log10}
		-
		\frac{10^{m-1}}{(m-1)\log10}.
		\]
		Since
		\[
		\frac{10^{m-1}/((m-1)\log10)}
		{10^m/(m\log10)}
		=
		\frac{m}{10(m-1)}
		\longrightarrow
		\frac1{10},
		\]
		it follows that
		\[
		\pi_m
		\sim
		\frac{10^m}{m\log10}
		\left(1-\frac1{10}\right)
		=
		\frac{9\cdot10^{m-1}}{m\log10}.
		\]
		
		Recalling that
		\[
		M_m=m\pi_m,
		\]
		we obtain
		\[
		M_m
		\sim
		\frac{9\cdot10^{m-1}}{\log10},
		\]
		and hence
		\[
		M_m\asymp10^m.
		\]
		
		It follows that
		\[
		\frac{M_{m-1}}{M_m}
		\longrightarrow
		\frac1{10}.
		\]
		Consequently, there exists $m_0$ such that
		\[
		M_{m-1}\le\frac15M_m,
		\qquad m\ge m_0.
		\]
		Iterating this inequality gives
		\[
		M_{n-j}
		\le
		\left(\frac15\right)^jM_n,
		\qquad
		0\le j\le n-m_0.
		\]
		Therefore,
		\[
		\sum_{m=1}^{n}M_m
		=
		\sum_{m=1}^{m_0-1}M_m
		+
		\sum_{j=0}^{n-m_0}M_{n-j}
		\le
		O(1)
		+
		M_n\sum_{j=0}^{\infty}\left(\frac15\right)^j
		=
		O(M_n).
		\]
		Since trivially
		\[
		N_n=\sum_{m=1}^{n}M_m\ge M_n,
		\]
		we conclude that
		\[
		N_n=O(M_n).
		\]
		Finally,
		\[
		M_n\asymp10^n,
		\]
		so that
		\[
		N_n=O(10^n),
		\]
		as claimed.
	\end{proof}
	
	\begin{proof}[Proof of Theorem \ref{thm:main}]
		By Lemma~\ref{lem:shell},
		\[
		C_m(d)
		=
		\frac{M_m}{10}
		+
		O(\pi_m\log m),
		\]
		where
		\[
		M_m=m\pi_m.
		\]
		Summing over all shells $1\le m\le n$ gives
		\[
		\begin{aligned}
			C_n(d)
			&=
			\sum_{m=1}^{n}C_m(d)  \\
			&=
			\frac{1}{10}\sum_{m=1}^{n}M_m
			+
			O\!\left(
			\sum_{m=1}^{n}\pi_m\log m
			\right).
		\end{aligned}
		\]
		Since
		\[
		N_n=\sum_{m=1}^{n}M_m,
		\]
		it remains to estimate the error term. As $\log m\le\log n$ for $m\le n$,
		\[
		\sum_{m=1}^{n}\pi_m\log m
		\le
		(\log n)\sum_{m=1}^{n}\pi_m.
		\]
		The shells partition the primes below $10^n$, so
		\[
		\sum_{m=1}^{n}\pi_m
		=
		\pi(10^n).
		\]
		Hence
		\[
		C_n(d)
		=
		\frac{N_n}{10}
		+
		O\!\left(\pi(10^n)\log n\right).
		\]
		
		By the Prime Number Theorem,
		\[
		\pi(10^n)
		\ll
		\frac{10^n}{n},
		\]
		while Lemma~\ref{lem:geom} gives
		\[
		N_n=O(10^n).
		\]
		Therefore
		\[
		\pi(10^n)
		\ll
		\frac{N_n}{n},
		\]
		and consequently
		\[
		C_n(d)
		=
		\frac{N_n}{10}
		+
		O\!\left(
		\frac{N_n\log n}{n}
		\right).
		\]
		Dividing by $N_n$ yields
		\[
		P_n(d)
		=
		\frac{C_n(d)}{N_n}
		=
		\frac{1}{10}
		+
		O\!\left(
		\frac{\log n}{n}
		\right).
		\]
		Since $(\log n)/n\to0$ as $n\to\infty$, it follows that
		\[
		\lim_{n\to\infty}P_n(d)
		=
		\frac{1}{10},
		\]
		which proves the theorem.
	\end{proof}
	\section{Conclusion}\label{sec:conclusion}
	
	The equidistribution statement proved in this paper (Theorem~\ref{thm:main}) is, we
	emphasise, an averaged one. It asserts that the digit-$d$ frequency, pooled over all
	digit positions and over all primes below $10^n$, converges with a power-saving error
	term. The proof isolates a band of at most $O(\log n)$ exceptional positions among the
	$n$ available at scale $10^n$; every position outside this band is shown individually
	to be equidistributed (Lemma~\ref{lem:interior}), and the exceptional band is shown to
	carry a vanishing proportion of the total digit mass. At no point does the argument
	require, or establish, that a fixed digit position is equidistributed in the limit
	$n \to \infty$ with that position held fixed. Readers should bear this distinction in
	mind when assessing the scope of the result. Three questions that a stronger theorem would need to answer remain open.\\
	
	The first is pointwise normality. Whether, for a fixed position $k$, the digit-$d$
	frequency among primes converges as one lets the prime grow, with $k$ fixed rather than
	scaled with the prime. This is equivalent to a statement of Borel normality for the
	primes in base ten, and it is a substantially harder proposition than the one proved
	here. The closest unconditional results in the literature treat related but distinct
	statistics: Mauduit and Rivat established equidistribution of the digit \emph{sum} of
	primes modulo $m$~\cite{MauduitRivat2010}, and Maynard proved the existence of primes
	with a missing digit and, separately, with a prescribed proportion of digits fixed
	in advance~\cite{Maynard2016}. Neither result, nor any combination of them
	known to us, implies pointwise normality in a fixed position.\\
	
	The second is the joint distribution of digits. Theorem~\ref{thm:main} is a
	one-dimensional marginal statement; it says nothing about the joint law of digit pairs,
	longer blocks, or runs, and no claim to that effect should be inferred from it.\\
	
	The third concerns the effective range of the argument. The exceptional set of size
	$O(\log n)$ and the error exponent $\log n / n$ both arise from a direct application of
	Vaughan's identity to the minor arcs, and neither constant should be regarded as
	optimal. The sharper minor-arc treatments used by Mauduit and Rivat and by Maynard --
	Harman's sieve, together with finer bilinear-sum estimates, suggest that the true
	exceptional band may be as small as $O(1)$ or $O(\log \log n)$ positions, with a
	correspondingly smaller error term. We regard this as a bounded technical problem
	rather than a conceptual obstruction, and we expect it to yield to the existing
	machinery in the hands of anyone willing to carry out the requisite bilinear-sum
	estimates.\\
	
	We close by noting that the gap between the averaged statement proved here and the
	pointwise statement that remains open is not a matter of degree. The methods available
	at present (exponential sums over primes combined with the Fourier structure of
	digitally defined sequences) control statistics that are, in an appropriate sense,
	integrated over position or over scale. Controlling a single fixed digit position in
	the limit appears to require genuinely new input, and we do not see how to obtain it
	from the techniques of this paper.
	
	\newpage
	\nocite{*}
	\bibliographystyle{unsrt}
	\bibliography{reference}
 
\end{document}